\documentclass[11pt]{amsart}
\usepackage{comment}

\usepackage[utf8]{inputenc}
\usepackage[T1]{fontenc}
\DeclareMathAlphabet{\mathpzc}{OT1}{pzc}{m}{it}
\usepackage{geometry}

\usepackage[dvipsnames]{xcolor}

\usepackage{xstring}
\newcommand{\mh}[2][e]{%
    \IfStrEqCase{#1}{%
        {e}{\textcolor{Blue}{\textbf{*Mary: #2*}}}%
        {c}{}%
    }
}
\newcommand{\nd}[2][e]{%
    \IfStrEqCase{#1}{%
        {e}{\textcolor{PineGreen}{\textbf{*Nathan: #2*}}}%
        {c}{}%
    }
}
\newcommand{\mhadd}[2][e]{%
    \IfStrEqCase{#1}{%
        {e}{\textcolor{Blue}{\textbf{#2}}}%
        {r}{#2}%
    }
}
\newcommand{\ndadd}[2][e]{%
    \IfStrEqCase{#1}{%
        {e}{\textcolor{PineGreen}{\textbf{#2}}}%
        {r}{#2}%
    }
}

\usepackage{amsfonts}
\usepackage{amssymb}
\usepackage{amsthm}
\usepackage{amsmath}
\usepackage{amscd}
\usepackage[shortlabels]{enumitem}
\usepackage{mathrsfs}
\usepackage{tikz}
\usepackage{pgfplots}
\pgfplotsset{compat=1.18}
\usepackage{float}
\usetikzlibrary{calc,arrows,decorations.pathreplacing}
\usepackage{nicefrac, xfrac}
\usepackage{mathtools,xparse}
\usepackage{soul}
\usepackage{appendix}

\usepackage[pagebackref, colorlinks = true,
linkcolor = red,
urlcolor  = blue,
citecolor = blue,
anchorcolor = blue]{hyperref}

\theoremstyle{plain}

\newtheorem{theorem}{Theorem}[section]
\newtheorem{corollary}[theorem]{Corollary}
\newtheorem{lemma}[theorem]{Lemma}

\newtheorem{proposition}[theorem]{Proposition}

\theoremstyle{definition}
\newtheorem{definition}[theorem]{Definition}

\newtheorem{remark}[theorem]{Remark}

\newtheorem*{theorem*}{Theorem}

\renewcommand{\epsilon}{\varepsilon}

\newcommand{\vertiii}[1]{{\left\vert\kern-0.25ex\left\vert\kern-0.25ex\left\vert #1
		\right\vert\kern-0.25ex\right\vert\kern-0.25ex\right\vert}}

\newcommand{\om}{\ensuremath{\omega}}

\DeclareSymbolFont{bbold}{U}{bbold}{m}{n}
\DeclareSymbolFontAlphabet{\mathbbold}{bbold}

\newcommand{\cG}{\ensuremath{\mathcal{G}}}

\newcommand{\cP}{\ensuremath{\mathcal{P}}}

\newcommand{\cX}{\ensuremath{\mathcal{X}}}
\newcommand{\cY}{\ensuremath{\mathcal{Y}}}
\newcommand{\cZ}{\ensuremath{\mathcal{Z}}}

\newcommand{\NN}{\ensuremath{\mathbb N}}

\newcommand{\QQ}{\ensuremath{\mathbb Q}}
\newcommand{\RR}{\ensuremath{\mathbb R}}

\newcommand{\ZZ}{\ensuremath{\mathbb Z}} 

\providecommand{\phantomsection}{}
\AtBeginDocument{\let\textlabel\label}
\makeatletter
\newcommand{\mylabel}[2]{\raisebox{.7\normalbaselineskip}{\phantomsection}(#1)%
	\def\@currentlabel{#1}\textlabel{#2}}
\makeatother

\makeatletter
\newcommand\xlabel[2][]{\phantomsection\def\@currentlabelname{#1}\label{#2}}
\makeatother

\ExplSyntaxOn
\NewDocumentCommand{\mathlist}{ O{,} m m }
 {
  \egreg_mathlist:nnn { #1 } { #2 } { #3 }
 }

\seq_new:N \l__egreg_mathlist_seq
\cs_new_protected:Npn \egreg_mathlist:nnn #1 #2 #3
 {
  \seq_set_split:Nnn \l__egreg_mathlist_seq { #1 } { #3 }
  \seq_use:Nnnn \l__egreg_mathlist_seq { #2 } { #2 } { #2 }
 }
\ExplSyntaxOff

\allowdisplaybreaks

\numberwithin{equation}{section}
\title[]{Complex Diophantine approximations and cusp excursions}
\date{\today}

\author[N. Dalaklis]{Nathan Dalaklis}
\address[N. Dalaklis]{Department of Mathematics, University of Oklahoma, Norman, OK 73019, USA}
\email{\href{ndalaklis@ou.edu}{ndalaklis@ou.edu}}

\author[Y.M. He]{Yan Mary He}
\address[Y.M. He]{Department of Mathematics, University of Oklahoma, Norman, OK 73019, USA}
\email{\href{he@ou.edu}{he@ou.edu}}

\begin{document}

\begin{abstract}
We study the Hausdorff dimension spectrum of asymptotic approximation rates of complex Diophantine approximation and
that of the asymptotic average excursion time of cusp excursions on the Bianchi orbifold $\mathbb H^3/\operatorname{PSL}(2,\mathbb Z[i])$ via a unified approach using the Hurwitz map. In particular, we construct a conformal graph directed system (CGDS) for the Hurwitz map and show that the Lyapunov exponent of the Hurwitz CGDS simultaneously captures
the asymptotic approximation rate and the the asymptotic average excursion time. 
Applying the multifractal analysis of Lyapunov exponents for this system,
we obtain a formula and
real-analyticity for the Hausdorff dimension spectrum functions.
\end{abstract}
	
\maketitle

\section{Introduction}
Classical Diophantine approximation studies the extent to which real numbers can be approximated by rational numbers. In particular, the quality of rational approximation is encoded by continued fractions: if $p_n/q_n$ denotes the $n$-th convergent of a real continued fraction, then the exponential behavior of the error
$\left|x-\frac{p_n}{q_n}\right|$
is governed by the growth rate of the denominators $q_n$. This relationship connects number theory, symbolic dynamics, and hyperbolic geometry.

One of the main goals of this paper is to study an analogous problem in the complex setting. We consider Hurwitz complex continued fractions, where the partial quotients lie in the Gaussian integers $\mathbb Z[i]$, and we investigate the exponential rate of complex Diophantine approximation by Gaussian rationals. 
The Hurwitz complex continued fraction algorithm produces convergents $p_n/q_n\in\mathbb Q(i)$. The first main goal of this paper is to study the level sets on which the approximation error
$\left|z-\frac{p_n}{q_n}\right|$
has a prescribed exponential decay rate.

\smallskip

On the other hand, the study of {\it cusp excursions} in hyperbolic geometry
provides a way to measure the recurrence of geodesic trajectories to the non-compact part of a hyperbolic orbifold, which quantifies how geodesics distribute themselves in a finite-volume but non-compact hyperbolic space. In particular, consider
the Bianchi orbifold $\mathcal O = \mathbb H^3/ \operatorname{PSL}(2,\mathbb Z[i])$. This orbifold is non-compact and has a cusp corresponding to infinity. A geodesic on $\mathcal O$ may travel arbitrarily far into this cusp before returning to the compact part of the orbifold. Such behavior is called a cusp excursion. Following Baumgartner--Pollicott \cite{baumgartnerP}, one can define an {\it asymptotic average excursion time} for $\gamma$. The second main goal of this paper is to study the level sets of (lifts of) geodesics having a prescribed asymptotic average excursion time.

We study these two problems via a {\it unified approach} using the Hurwitz map $T:X\to X$
given by
$$T(z) =
    \begin{cases}
    \displaystyle \frac{1}{z}
    -
    f\!\left(\frac{1}{z}\right),
    & z\neq 0, \\[1em]
    0, & z=0,
    \end{cases}
$$
where $X
    :=
    \{x+iy\in\mathbb C : x,y\in[-1/2,1/2)\}$ and $f(w)$ denotes the nearest Gaussian integer to $w$, namely
$f(w):=\left\lfloor \operatorname{Re}(w)+\frac12\right\rfloor
    +
    i\left\lfloor \operatorname{Im}(w)+\frac12\right\rfloor.
$
The Hurwitz map generates the complex continued fraction algorithm, in the same way that the Gauss map generates the classical continued fraction algorithm for real numbers. Moreover, the Hurwitz complex continued fraction algorithm is closely related to, but distinct from, the classical real continued fraction algorithm: when restricted to real numbers, its convergents form a subsequence of the real continued fraction convergents, with controlled omissions and comparable Diophantine approximation properties; see \cite[Theorem 2.1]{Sim17}.

We encode the Hurwitz map by a {\it conformal graph directed system} (CGDS) in the sense of Urba\'nski, Roy and Munday (see \cite[Definition 19.3.1]{Urbanski22}). The Lyapunov exponent of this system simultaneously records the exponential rate of complex Diophantine approximation and the asymptotic average cusp excursion time. 
Our results are then obtained by applying the {\it multifractal analysis} of Lyapunov exponents developed in \cite{NDDis} for cofinitely regular finitely irreducible CGDS. 

We note that Mauldin and Urba\'nski introduced and studied the fractal geometry of a complex continued fraction algorithm on the ball of radius $\frac{1}{2}$ centered at $\frac{1}{2}$ in the complex plane in \cite{MU1996} using a {\it conformal iterated function system} (CIFS). Later in \cite{CLU19}, Chousionis, Leykekhman, and Urba\'nski refined this analysis by studying the dimension spectrum of the limit set of this CIFS. Although the dimension spectrum and Lyapunov spectrum are often closely related, our analysis is distinct as the Hurwitz CGDS is not a CIFS.

\subsection{Statement of results}
Let $z \in \mathbb C$ and consider the Hurwitz complex continued fraction of $z$ (i.e., continued fraction expansion of $z$ using Gaussian integers $\mathbb Z[i]$). Let $p_n/q_n \in \mathbb Q(i)$ be Hurwitz convergents of $z$; see Section 3 for details. For $\alpha \in \mathbb R$, we study level sets of the {\it asymptotic approximation rate}
$$
X(\alpha):=
    \left\{
    z \in X :
    -\lim_{n\to\infty}
    \frac{1}{n}
    \log
    \left|
    z-\frac{p_n}{q_n}
    \right|
    =
    \alpha
    \right\}.
$$
Our first main theorem gives a formula and analyticity for the Hausdorff dimension of these level sets.

\begin{theorem}\label{thm:main-1}
Let $\lambda_{\min}$ denote the minimal Lyapunov exponent of the Hurwitz conformal graph directed system. For every $\alpha>\lambda_{\min}$, the Hausdorff dimension $S(\alpha):=\operatorname{HD}(X(\alpha))$
is given by
$$
    S(\alpha)
    =
    \frac{1}{\alpha}
    \left(\mathcal P(-q(\alpha))-q(\alpha)\right)
    =
    \frac{1}{\alpha}
    \inf_{q\in\mathbb R}
    \{\mathcal P(-q)-q\alpha\}.
$$
Moreover, the function $\alpha\mapsto S(\alpha)$ is real-analytic on $(\lambda_{\min},\infty)$.
\end{theorem}

Here $\mathcal P$ is the pressure associated to the geometric potential of the Hurwitz coding, and $q(\alpha)$ is determined by the corresponding pressure equations. See Theorem \ref{thm:main-3} for details.

\medskip

Let $\mathcal O:=\mathbb H^3/\operatorname{PSL}(2,\mathbb Z[i])$ and let $\gamma \subset \mathcal O$ be a geodesic issuing from the cusp.
Following Baumgartner--Pollicott \cite{baumgartnerP}, one can define an {\it asymptotic average excursion time} for $\gamma$ as follows. If $t_n^*(\gamma)$ denotes the corresponding sequence of maximal excursion times, then the asymptotic average excursion time is
given by $\lim_{n\to\infty}\frac{t_n^*(\gamma)}{n}$,
whenever the limit exists. 

Our second main theorem gives a formula and real-analyticity for the Hausdorff dimension of level sets of asymptotic average excursion time.
Let $z \in X$ and let $\gamma_z$ be a lift of a geodesic in $\mathcal{O}$ such that $\gamma_z$ has endpoints $(z,\infty)$.
Set 
$$G(\alpha)
    :=
    \left\{
    z \in X :
    \lim_{n\to\infty}
    \frac{t_n^*(\gamma_z)}{n}
    =
    \alpha
    \right\}.$$

\begin{theorem}\label{thm:main-2}
For every $\alpha>\lambda_{\min}$, the Hausdorff dimension $\hat S(\alpha):=\operatorname{HD}(G(\alpha))$
is given by
$$
    \hat S(\alpha)
    =
    \frac{1}{\alpha}
    \left(\mathcal P(-q(\alpha))-q(\alpha)\right)
    =
    \frac{1}{\alpha}
    \inf_{q\in\mathbb R}
    \{\mathcal P(-q)-q\alpha\}.
$$
Moreover, the function $\alpha \mapsto \hat S(\alpha)$ is real analytic on $(\lambda_{\min},\infty)$.
\end{theorem}

\subsection{Strategies of the proofs}
We proceed symbolically. For any $\om$ in the symbolic space $E_A^\infty$ of the Hurwitz CGDS, $\pi(\om)$ is a Gaussian irrational number in $X$, where $\pi : E_A^\infty \to J$ is the projection map and $J\subset X$ is the limit set of the Hurwitz CGDS.

Denote by $\lambda(\om)$ the Lyapunov exponent of $\om$. 
The connection between $\lambda(\om)$, the asymptotic approximation rate and the asymptotic average excursion time is given by the following identity
\begin{align*}
\lambda(\omega)= 2\lim_{n\to\infty}\frac{1}{n}\log |q_n|
&=-\lim_{n\to\infty} \frac{1}{n}
    \log
    \left|
    \pi(\omega)-\frac{p_n}{q_n}
    \right|\\
&=\lim_{n\to\infty}
    \frac{t_n^*(\pi(\omega))}{n}.
\end{align*}
Thus the Lyapunov multifractal decomposition of the Hurwitz CGDS gives a unified framework for studying both complex Diophantine approximation and cusp excursions. We obtain a description of the Lyapunov spectrum in terms of the associated pressure function by applying the general theory of multifractal analysis of Birkhoff averages for confinitely regular finitely irreducible CGDMS satisfying the strong open set condition developed in \cite{NDDis}. In particular, our symbolic approach to the multifractal analysis on $J$ is enough for our conclusions by \cite[Proposition 6.5]{NDDis} and the countable stability of Hausdorff dimension.

\subsection{Organization of the paper}
The paper is organized as follows. In Section 2 we construct the Hurwitz conformal graph directed system and establish the thermodynamic properties needed later. In Section 3 we relate Lyapunov exponents to complex Diophantine approximation and prove the first main theorem. In Section 4 we apply the same Lyapunov analysis to cusp excursions on the Bianchi orbifold and prove the second main theorem.

\section{Thermodynamic formalism}
In this section, we construct a conformal graph directed system (CGDS) associated to the Hurwitz map, which is the thermodynamic framework used throughout the paper. We verify the properties of this CGDS needed to apply thermodynamic formalism: finite irreducibility, finite primitivity, the strong open set condition and cofinitely regularity. We then describe the associated pressure function, the Lyapunov pressure, and the equilibrium measures that will be used later in the multifractal analysis of approximation rates and cusp excursions.

\subsection{The Hurwitz map}
Let
$$
X := \{x+iy \in \mathbb{C} : x,y \in [-\tfrac12,\tfrac12) \}
$$ 
be the unit square in the complex plane centered at the origin. 
Recall that the \emph{Hurwitz map} $T \colon X \to X$ is defined as
$$
T(z) := \begin{cases}
\frac{1}{z} - f(\frac{1}{z}), {\rm if ~} z \neq 0;\\
0, {\rm if~} z = 0.
\end{cases}
$$
where $f(w)$ denotes rounding $w \in \mathbb{C}$ to the nearest Gaussian integer in $\mathbb{Z}[i]$, i.e., 
$f(w):=\lfloor \operatorname{Re}(w)+\frac12\rfloor
+i\lfloor\operatorname{Im}(w)+\frac12\rfloor.$
Denote by
$$D := \ZZ[i]\setminus \{0,\pm1,\pm i\}.$$
The inverse branches of $T$ are indexed by $a \in D$, namely, $\phi_a(z) :=  \frac{1}{a+z}$.

\subsection{The Hurwitz CGDS}
The goal of this section is to construct a conformal graph directed system in the sense of 
Urba\'nski, Roy and Munday (see \cite[Definition 19.3.1]{Urbanski22}) for the Hurwitz map.

We begin with the prototype sets. For any given $a\in D$ there exists a maximal closed region $Q_a\subseteq \overline{X}$ such that $\phi_a\vert_{Q_a}(Q_a)\subseteq \overline{X}$. We call these $Q_a$ sets the {\it prototype sets}. In many cases, $Q_a = \overline{X}$, but this need not be the case. As described in \cite[Remark 5.8]{BGH23} there are $13$ such sets.  Now let 
$$
V=\cX := \{C_1,C_2,C_3,C_4,H_1,H_{i},H_{-1},H_{-i},P_1,P_2,P_3,P_4\}
$$
be the collection of regular closed sets given by the disjointification of these $13$ prototype sets. These sets can be seen in Figure \ref{fig:1} and they weakly partition $\overline{X}$ as they overlap on their boundaries. 

\begin{figure}[H]
\begin{tikzpicture}
    \begin{axis}[
        axis lines=center,
        label style={font=\tiny},
        xlabel={${\rm Re}(z)$},
        ylabel={${\rm Im}(z)$},
        xmin=-.7, xmax=.7,
        ymin=-.7, ymax=.7,
        ticks = none,
        grid=none,
        width=8cm, height=8cm,
    ]
        \addplot[
            domain=0:1, 
            samples=100,
            smooth,
            thick,
            blue
        ] ({cos(deg(-pi/3 - pi/3*x))}, {sin(deg(-pi/3 - pi/3*x)) + 1});
        \addplot[
            domain=0:1, 
            samples=100,
            smooth,
            thick,
            blue
        ] ({cos(deg(pi/3 + pi/3*x))}, {sin(deg(pi/3 + pi/3*x)) - 1});
        \addplot[
            domain=0:1, 
            samples=100,
            smooth,
            thick,
            blue
        ] ({cos(deg(-pi/6 + pi/3*x))-1}, {sin(deg(-pi/6 + pi/3*x))});
        \addplot[
            domain=0:1, 
            samples=100,
            smooth,
            thick,
            blue
        ] ({cos(deg(5*pi/6 + pi/3*x))+1}, {sin(deg(5*pi/6 + pi/3*x))});
        \addplot[
            domain=0:1, 
            samples=100,
            smooth,
            thick,
            blue
        ] ({cos(deg(pi/6 + pi/6*x))-1}, {sin(deg(pi/6 + pi/6*x))-1});
        \addplot[
            domain=0:1, 
            samples=100,
            smooth,
            thick,
            blue
        ] ({cos(deg(-pi/3 + pi/6*x))-1}, {sin(deg(-pi/3 + pi/6*x))+1});
        \addplot[
            domain=0:1, 
            samples=100,
            smooth,
            thick,
            blue
        ] ({cos(deg(-2*pi/3 - pi/6*x))+1}, {sin(deg(-2*pi/3 - pi/6*x))+1});
        \addplot[
            domain=0:1, 
            samples=100,
            smooth,
            thick,
            blue
        ] ({cos(deg(-4*pi/3 + pi/6*x))+1}, {sin(deg(-4*pi/3 + pi/6*x))-1});
        \addplot[
            domain=0:1,
            samples=100,
            smooth,
            thick,
            blue
        ] ({.5},{x-.5});
        \addplot[
            domain=0:1,
            samples=100,
            smooth,
            thick,
            blue
        ] ({x-.5},{.5});
        \addplot[
            domain=0:1,
            samples=100,
            smooth,
            thick,
            blue
        ] ({-.5},{x-.5});
        \addplot[
            domain=0:1,
            samples=100,
            smooth,
            thick,
            blue
        ] ({x-.5},{-.5});

        \addplot[black,no marks] coordinates{(-.5,-.5)}
            node[above right] {$C_3$};
        \addplot[black,no marks] coordinates{(.5,-.5)}
            node[above left] {$C_4$};
        \addplot[black,no marks] coordinates{(.5,.5)}
            node[below left] {$C_1$};
        \addplot[black,no marks] coordinates{(-.5,.5)}
            node[below right] {$C_2$};

        \addplot[black,no marks] coordinates{(0,-.5)}
            node[above] {$H_{i}$};
        \addplot[black,no marks] coordinates{(0,.5)}
            node[below] {$H_{-i}$};
        \addplot[black,no marks] coordinates{(.5,0)}
            node[left] {$H_1$};
        \addplot[black,no marks] coordinates{(-.5,0)}
            node[right] {$H_{-1}$};

        \addplot[black,no marks] coordinates{(.2,-.2)}
            node {$P_4$};
        \addplot[black,no marks] coordinates{(.2,.2)}
            node {$P_1$};
        \addplot[black,no marks] coordinates{(-.2,.2)}
            node {$P_2$};
        \addplot[black,no marks] coordinates{(-.2,-.2)}
            node {$P_3$};
        
    \end{axis}
\end{tikzpicture}
\caption{The sets in $\cX$.}
        \label{fig:1}
\end{figure}
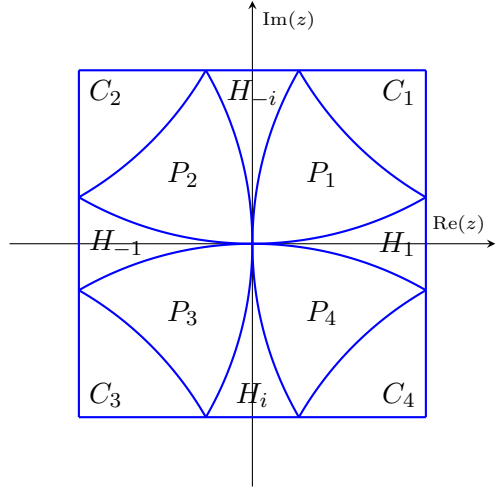

By definition of the prototype sets, for every $a\in D $ there exists 
$\cY(a)\subseteq \cX$ such that for all $Y\in \cY(a)$, the inversion $\phi_a$ is such that $\phi_a(Y)\subseteq \overline{X}$. When $\cY(a)\neq \cX$, the sets in $\cY(a)$ are given in Table \ref{tab:1}. 
\begin{lemma}
    Given $a\in D$ and $Y\in \cY(a)$ there is a unique $W\in \cX$ such that $\phi_{a}\vert_Y(Y)\subseteq W$. 
\end{lemma}
\begin{proof}
Observe that for each $a\in D$, there exists $\cZ(a)\subseteq \cX$ such that $\phi_a(Q_a)\subset\bigcup \cZ(a)$ as seen in Figure \ref{fig:Cylinders}. For $|a|>\sqrt{8}$, $|a| \leq 2$ , $\cZ(\alpha)$ is a singleton by Propositions \ref{edges1}, \ref{edges2}, \ref{edges4} and \ref{edges5}. So $W$ is unique and we are done. 

If not, then $\cZ(\alpha)$ contains exactly $2$ sets, $P_k$ and $C_k$ for some $k$. By Proposition \ref{edges3} and Table \ref{tab:2}, each $Y\in \cY(a)$ maps uniquely into one of these sets.
\end{proof}

The previous lemma defines an edge set $E\subset D\times \cX$ between the vertices in $V$. That is, we say there exists a directed edge $e$ with $i(e) = W$ and $t(e)=Y$ if and only if there exists $a \in D$ such that for the triple $(a,Y,W)\in D\times \cX\times \cX$ the restriction of the inversion $\phi_a\vert_Y$ is such that $\phi_a\vert_Y(Y)\subseteq W$. In this case, we say $e=(a,Y)$, and we define $X_{t(e)} := Y$ and $X_{i(e)}:= W$.

We define the admissibility matrix $A:E\times E\to\{0,1\}$ by
$A_{ef}=1$ if and only if $t(e)=i(f)$. In particular, we will use the adjacency rule of the underlying graph $\cG = (V,E,i,t)$. For each $e = (e_1,e_2)\in E$, there is an associated restricted inversion $\phi_e:X_{t(e)}\to X_{i(e)}$ given by $\phi_e := \phi_{e_1}\vert_{e_2}$. Then, we take $\Phi:=\{\phi_e\}_{e\in E}$.

For this directed multigraph, $\cG$, with admissibility rule, $A$, and the collection of restricted inversions, $\Phi$, we consider the tuple
$$
S=(\Phi,\cG,A).
$$

We will show shortly that $S$ is a specific type of conformal graph directed Markov system (CGDMS). CGDMSs were introduced by Urba\'nski and Mauldin in \cite{MU03}, where they require an additional cone condition. However, we will instead proceed with the definition given by Urba\'nski, Roy and Munday in \cite[Definition 19.3.1]{Urbanski22}. The statement of the definition of these systems is quite lengthy, but we recall it for the reader's convenience. We begin with the notion of a graph directed Markov system (GDMS).

\begin{definition} 
    A GDMS is a triple $(\Phi, \cG,A)$ where $\Phi = \{\phi_e\}_{e\in E}$ is a collection of maps between a finite family of compact metric spaces $\{X_v\}_{v\in V}$, $G = (V,E,i,t)$ is a directed multigraph, and $A:E\times E\to \{0,1\}$ is a matrix with $A_{ef} = 1$ if and only if $\phi_e\circ\phi_f:X_{t(f)}\to X_{i(e)}$ is allowable and defined.  
\end{definition}

Every GDMS is a subsystem of its underlying {\it graph directed system (GDS)} which is the system where $A_{ef}=1$ if and only if $t(e)=i(f)$. A GDS is an {\it IFS} if $|V|=1$ and $A_{ef} = 1$ for all $e,f\in E$. We call a GDS \textit{proper} if it is not an IFS or a subsystem of an IFS.

\begin{definition}\label{def:con}
 A GDMS is called a {\it conformal GDMS} (CGDMS), if its underlying GDS satisfies the following properties
 \begin{itemize}
    \item[1.] For every $v\in V$, there is a $d$ such that $X_v\subset \RR^d$ is regular closed.\\
    \item[2.] $\Phi$ satisfies the open set condition,
    $$
        \phi_{e_1}({\rm Int}(X_{t(e_1)}))\cap\phi_{e_2}({\rm Int}(X_{t(e_2)})) = \emptyset.
    $$
    \item[3.] For every $v\in V$ there is an open, connected and bounded set $W_v$ such that $X_v\subseteq W_v\subseteq \RR^d$ and such that for every $e\in E$ with $t(e) = v$, the map $\phi_e$ extends to a contracting conformal diffeomorphism of $W_v$ into $\RR^d$, and all fo these extensions have a common contraction ratio $s<1$. 
    \item[4.] There exists constants $L\geq 1$ and $\kappa>0$ such that 
    $$
        \left\vert|\phi_e'(y)|-|\phi_e'(x)|\right|\leq L|\phi'_e(x)|\cdot ||y-x||^{\kappa}
    $$
    for every $e\in E$ and every pair of points $x,y\in W_{t(e)}$.
 \end{itemize}
\end{definition}

\begin{definition} 
    We say a CGDMS is \textit{finitely irreducible} if there exists a finite list of admissible paths (or words) $I$ such that for every $(a,b)\in E\times E$ there is a $\tau(a,b)\in I$ such that $a\tau(a,b)b$ is admissible. If all paths (or words) in $I$ are the same length we say the system is \textit{finitely primitive}.
\end{definition}

\begin{theorem}\label{thm:FI}
    $S$ is a proper conformal graph directed system (CGDS) (i.e., not a conformal iterated function system) and finitely irreducible.
\end{theorem}

\begin{proof}
    From the construction of $S$ at the beginning of this section, it is clear that $S$ is a GDS. 
    
    For conformality, we note that for the collection of maps $\Phi:=\{\phi_e\}_{e\in E}$ and a fixed $0<\eta<1$ small enough, we have a common contraction ratio of $s= \frac{1}{2-\eta}$ on a bounded open neighborhood $W_{t(e)}\supset X_{t(e)}\in \cX$ on which the map $\phi_e$ is a contracting conformal diffeomorphism. $\Phi$ also satisfies the open set condition (see Figure \ref{fig:Cylinders}). Further by \cite[Remark 19.3.2(c)]{Urbanski22}, there exists $L\geq 1$ such that 
    $$
        \left\vert|\phi_e'(y)|-|\phi_e'(x)|\right|\leq L|\phi'_e(x)|\cdot |y-x|
    $$
    for every $e\in E$ and pair of points $x,y\in W_{t(e)}$. Hence the four properties  required of the underlying GDS in Definition \ref{def:con} are satisfied, and $S$ is a CGDS.
    
    For properness, the case analysis in Propositions \ref{edges1}-\ref{edges6} demonstrates that it is necessary that $|V|>1$ and that under the adjacency rule $A$ there exists $V_1,V_2\in V$ such that there is no edge $e$ such that $t(e) = V_1$ and $i(e)=V_2$. So $S$ is not an IFS nor is it a subsystem of an IFS as desired. 
    
    For finite irreducibility of the Hurwitz CGDS, we construct a finite words list $I$. Consider $((a,Y),(b,Z))\in E\times E$ where $\phi_{(a,Y)}:Y\to W$ and $\phi_{(b,Z)}: Z\to V$. We proceed by cases on $Z$. 
    
    If $Z\neq C_k$ for all $k\in \{1,2,3,4\}$ we may choose for membership in $I$ an appropriate $e\in E$ with $|e_1| \in\{3, \sqrt{10}\}$ such that $\phi_e: W\to Z$ by Propositions \ref{edges1} and \ref{edges2}. So $(a,Y)e(b,Z)$ is a valid path. 
    
    Next, if $Z= C_k$ for some $k$, we first pick $e\in E$ as in the previous case so that $\phi_e$ maps $W$ to $P_k$. Then, take $f\in E$ by the rule
    $$
    \begin{cases}
        f_1 = 1-2i &{\rm if \;} k=1\\
        f_1 = -1-2i &{\rm if \;} k=2\\
        f_1 = -1+2i &{\rm if \;} k=3\\
        f_1 = 1+2i &{\rm if \;} k=4
    \end{cases}
    $$
    Such a choice of $f$ exists by Proposition \ref{edges6}. For this choice of $f$, $(a,Y)ef(b,Z)$ is a valid path and so we include $ef$ in $I$.
    
    Hence $I$ is a subset of $\{e\in E\;\vert\; |e_1| \in\{3, \sqrt{10}\}\}\cup \{ef\in E^2\;|\; |e_1|\in\{3, \sqrt{10}\},|f_1| =\sqrt{5}\}$ which is a finite set, so $I$ is finite and witnesses the finite irreducibility of the Hurwitz CGDS as desired.  
\end{proof}

We call $S$ the Hurwitz CGDS. The symbolic representation of the Hurwitz CGDS is given by
$$
    E_{A}^{\infty}:= \{\om\in E^{\infty}\:|\; A(\om_i,\om_{i+1}) = 1, \forall i\in \NN\}
$$
Often times we will refer to $S$ by $\Phi$ alone where the symbolic representation $E_A^{\infty}$ is inferred without reference to $\cG$. The set of finite admissible words in the symbolic representation is $E_A^{*}$, and the set of finite admissible words of length $n$ is $E_A^{n}$. For any $\om\in E_A^{\infty}$, define $\om\vert_n:=\om_1\dots\om_n\in E_A^n$. For any finite admissible word $w\in E_A^n$ define $t(w) := t(w_n)$ and $i(w) := i(w_1)$. 

Then define
$$\phi_{w} = \phi_{w_1}\circ \dots \circ \phi_{w_n}:X_{t(w_n)}\to X_{i(\om_1)}.
$$

This is consistent with notation in \cite{Urbanski22} regarding conformal graph directed Markov systems of which conformal graph directed systems are a special case. 

\begin{corollary}
    The Hurwitz CGDS is finitely primitive, in fact, we may find $I$ witnessing the finite primitivity such that the length of each word in $I$ is exactly $2$.
\end{corollary}

\begin{proof}
    Take $I$ that was found in the argument of finite irreducibility. Then replace any word of length $1$, $(a,W)$, with $(a,W)(b,Z)$ where $|b|\in\{3, \sqrt{10}\}$ and $\phi_{(b,Z)}$ sends $Z$ to itself. Such a $(b,Z)$ exists by Propositions \ref{edges1} and \ref{edges2}.
\end{proof}

The diameters of images of maps in the Hurwitz CGDS satisfy the following estimate which is a consequence of \cite[Lemmas 19.3.9 and 19.3.11]{Urbanski22}.
    \begin{lemma}\label{lem:diam1}
        Let $\Phi = \{\phi_e\}_{e\in E}$ be the CGDS. Then, there exists a constant $D\geq 1$ and a bounded distortion constant $K\geq 1$ such that 
        $$
            D^{-1}\|\phi_{\om}'\|_{W_{t(\om)}}\leq {\rm diam}(\phi_{\om}(X_{t(\om)}))\leq KD\|\phi_{\om}'\|_{X_{t(\om)}}, \qquad \forall \om\in E_{A}^{*}.
        $$
    \end{lemma}

The limit set of the Hurwitz CGDS is given by 
$$
    J := \bigcup_{\om\in E_A^{\infty}}\bigcap_{n=1}^{\infty}\phi_{\om|_n}(X_{t(\om|_n)})
$$
and the coding or projection map $\pi:E_A^{\infty}\to J$ of the symbolic representation onto the limit set is defined by
$$
    \{\pi(\om)\} = \bigcap_{n=1}^{\infty}\phi_{\om|_n}(X_{t(\om|_n)}).
$$
Moreover, $\pi$ is the semi-conjugacy map between the natural shift map $\sigma: E_A^{\infty}\to E_A^{\infty}$ and $T\vert_J:J\to J$, and $J$ is the collection of Hurwitz irrationals $X\cap \QQ(i)^c$. 

\begin{definition} We say a CGDMS satisfies the {\it strong open set condition} (SOSC) if
$$
\bigcup_{v\in V}(J\cap {\rm Int} (X_{v})) \neq \emptyset.
$$
\begin{remark}\label{rmk:SOSC}
    The Hurwitz CGDS $\Phi$ satisfies the SOSC as $J\subseteq X$. 
\end{remark}

\end{definition}

\subsection{Pressure}

The pressure of a CGDMS is given by 
$$
    \cP(t):=\lim_{n\to\infty}\frac{1}{n}Z_n(t):=\lim_{n\to\infty}\frac{1}{n}\sum_{\om\in E_A^n}\||\phi'_{\om}|\|^t_{X_{t(\om)}}
$$
where 
$$
    \|\varphi\|_{X_{t(\om)}}:= \|\varphi\vert_{X_{t(\om)}}\|_{\infty}.
$$

The finiteness parameter $\theta$ of a CGDMS is given by 
$$
    \theta:= \inf\{t\in \RR:\cP(t)<\infty\}.
$$
When $\cP(\theta) = \infty$ we say that the CGDMS is cofinitely regular. By \cite[Proposition 19.4.5]{Urbanski22}, $P(t)<\infty$ if and only if $Z_1(t)<\infty$.

\begin{lemma}\label{lem:cfr}
    The Hurwitz CGDS is cofinitely regular with $\theta = 1$.
\end{lemma} 

\begin{proof}
    We will show that the first partition function of the Hurwitz CGDS, $Z_1(t)$, is comparable to the Dedekind zeta funciton of $\mathbb{Q}[i]$ given by $$\zeta_{Q[i]}(t)=\sum_{n,m=1}^{\infty}\left(\frac{1}{n^2+m^2}\right)^t,$$
    which for real $t$ converges if $t>1$. Implying that $\theta =1$. We begin with the upper bound. Here we have 

    \begin{align*}
        Z_1(t) := \sum_{e\in E}\||\phi_e'|\|^t_{X_{t(e)}} 
        &= \sum_{(a,Y)\in E}\left\|\left\vert\frac{1}{(z+a)^2}\right\vert\right\|_Y^t \geq \sum_{\substack{(a,Y)\in E\\ |a|\geq \sqrt{8}}} \left\|\left\vert\frac{1}{(z+a)^2}\right\vert\right\|_Y^t\\
        &\geq 12 \sum_{\substack{a\in \ZZ[i]\\ |a|\geq \sqrt{8}}}\inf\left\{\left\|\left\vert\frac{1}{(z+a)^2}\right\vert\right\|_Y^t\;\Bigg|\; Y\in \cX \right\}\\
        &= 12 \sum_{\substack{a\in \ZZ[i]\\ |a|\geq \sqrt{8}}} \left\vert\frac{1}{a^2}\right\vert^t\inf\left\{\left\|\left\vert\frac{a^2}{(z+a)^2}\right\vert\right\|_Y^t\;\Bigg|\; Y\in \cX \right\}.
    \end{align*}
    Now since all $Y \in \cX$ are compact, $\cX$ is a finite set, $|(z+a)|^2>1$, and since $\lim_{a\to\infty}\left\vert\frac{a^2}{(z+a)^2}\right\vert = 1$ we obtain that there is some $K,M>0$ such that for all $a\geq \sqrt 8$,
    $$
        M> K(a) := \inf\left\{\left\|\left\vert\frac{a^2}{(z+a)^2}\right\vert\right\|_Y^t\;\Bigg|\; Y\in \cX \right\} > K >0.
    $$
    And so,
    $$
        Z_1(t) \geq 12K\sum_{\substack{a\in \ZZ[i]\\ |a|\geq \sqrt{8}}}\left(\frac{1}{|a^2|}\right)^t = 12K(\zeta_{\QQ[i]}(t)-L).
    $$
    For some constant $L$. Thus, $\theta\leq 1$.\newline For the lower bound, 
    \begin{align*}
        Z_1(t) = \sum_{e\in E}\||\phi_e'|\|^t_{X_{t(e)}} 
        &= \sum_{(a,Y)\in E}\left\|\left\vert\frac{1}{(z+a)^2}\right\vert\right\|_Y^t = \sum_{\substack{(a,Y)\in E\\ |a|\geq \sqrt{8}}} \left\|\left\vert\frac{1}{(z+a)^2}\right\vert\right\|_Y^t+L_2 \\
        &\leq 12\sum_{\substack{a\in \ZZ[i]\\ |a|\geq \sqrt{8}}}\left(\frac{1}{|a^2|}\right)^t \sup\left\{\left\|\left\vert\frac{a^2}{(z+a)^2}\right\vert\right\|_Y^t\;\Bigg|\; Y\in \cX \right\}+L_2.
    \end{align*}
    By similar reasoning as in the previous case we obtain that there is some $K_1,M_1>0$ such that for all $a\geq \sqrt 8$,
    $$
        M_1> K_1(a) := \sup\left\{\left\|\left\vert\frac{a^2}{(z+a)^2}\right\vert\right\|_Y^t\;\Bigg|\; Y\in \cX \right\} > K_1 >0.
    $$ 
    And so 
    $$
        Z_1(t) \leq 12M_1\sum_{\substack{a\in \ZZ[i]\\ |a|\geq \sqrt{8}}}\left(\frac{1}{|a^2|}\right)^t +L_2 = 12M_1(\zeta_{\QQ[i]}(t)-L)+L_2.
    $$
    For the same constant $L$. Thus, $\theta\geq 1$ and so we conclude $\theta = 1$. Since $\zeta_{\QQ[i]}(\theta) = \infty$, and $Z_1(t)$ is comparable to $\zeta_{\QQ[i]}(t)$, $Z_1(\theta) = \infty$ which implies that $\mathcal P(\theta) = \infty$, and we conclude that the Hurwitz CGDS is cofinitely regular.
\end{proof}

We define the {\it multivariate Lyapunov pressure} by $\cP(t,q):= \cP(t-q)$. With this multivariate pressure we define the {\it Manhattan region} of the Lyapunov pressure by 
$$
    \mathsf{M} := \left\{(t,q)\in \RR^2: \cP(t,q)<\infty\right\}
$$
and call $\partial\mathsf{M}$ the corresponding {\it Manhattan curve}. The next corollary follows directly from Lemma \ref{lem:cfr}.

\begin{corollary}
For the multivariate Lyapunov pressure of the Hurwitz CGDS, $\partial \mathsf{M}$ is given by the line $L:q(t) = t-1$, and $\mathsf{M}$ is the half-plane $\{(t,q)\in \RR^2\;|\;q < t-1\}$. 
\end{corollary}

\subsection{Symbolic Pressure and Measures}

The geometric potential of a CGDMS is given by 
$$
    \Xi_{\Phi}(\om):= \log |\phi'_{\om_1}(\pi(\sigma(\om)))|,\;\forall\om\in E_A^{\infty}.
$$
Since $\Phi$ is a CGDS, by \cite[Lemma 19.4.1]{Urbanski22}, $\Xi_{\Phi}$ is acceptable. 
For any $\sigma$-invariant probability measure $\mu$ on $E_A^{\infty}$, the $\mu$-characteristic Lyapunov exponent of $\Phi$ is given by
$$
    \lambda(\sigma) := -\int_{E_A^{\infty}}\Xi_{\Phi}d\mu > 0.
$$
At a point $\om\in E_A^{\infty}$ we take 
$$
    \lambda(\om):= -\lim_{n\to\infty}\frac{1}{n}\sum_{j=0}^{n-1}\log|\phi'_{\sigma^j(\om)_1}(\pi(\sigma^{j+1}(\om)))|
$$
if the limit exists.
Take $f_{t,q}:= (t-q)\Xi_{\Phi}$. The pressure of this potential on the symbolic representation is given by 
$$
\cP(f_{t,q}) := \lim_{n\to\infty}\frac{1}{n}\log\sum_{\om\in E_A^n}\exp(\overline{S_{n}}f_{t,q}([\omega]))
$$
where for a set $\Lambda\subseteq E_A^{\infty}$
$$
    \overline{S_{n}}g(\Lambda) := \sup\{S_ng(\om):\om\in \Lambda \}.
$$
Since $\Phi$ is finitely irreducible, $\cP(f_{t,q})=\cP(t,q)$ by \cite[Proposition 19.4.6.(a)]{Urbanski22}. Further, since $\Phi$ is finitely primitive, $(t-q)\Xi_{\Phi}$ has a unique Gibbs state $\mu_{t,q}$ that is totally ergodic by \cite[Theorem 17.4.6]{Urbanski22}.

\begin{lemma}\label{lem:cylEntropy}
    For the Hurwitz CGDS $\Phi$, if $(t,q)\in \mathsf{M}$ we have $H_{\mu_{t,q}}(\mathcal{U}_E)<\infty$ where $\mathcal{U}_E$ is the partition of $E_A^{\infty}$ into its initial cylinders of length $1$. 
\end{lemma}

\begin{proof}
Let $\rho_e\in[e]$ and recall from the Gibbs property that,
    $$
        \mu_{t,q}([e])\asymp\exp((t-q)\Xi_{\Phi}(\rho_e)-\mathcal P(t,q)).
    $$
    Now we estimate $H_{\mu_{t,q}}(\mathcal{U}_E)$ using the Gibbs property.
    \begin{align*}
        H_{\mu_{t,q}}(\mathcal{U}_E) &= \sum_{e\in E}-\mu_{t,q}([e])\log\mu_{t,q}([e])\\
        &< \sum_{e\in E}-\mu_{t,q}([e])\log(C(\exp((t-q)\Xi_{\Phi}(\rho_e)-\mathcal P(t,q))))\\
        &=\mathcal P(t,q)+\log C-\sum_{e\in E} \mu_{t,q}([e])|\phi_{e}'(\pi(\sigma(\rho_e)))|^{t-q}.
    \end{align*}
    Since $(t,q)\in \mathsf{M}$, we have $\mathcal P(t,q)<\infty$. Thus, $Z_1(t-q)<\infty$ by \cite[Proposition 19.4.5]{Urbanski22}. Moreover,
    $$
    \sum_{e\in E}|\phi_{e}'(\pi(\sigma(\rho_e)))|^{t-q}<Z_1(t-q)<\infty.
    $$
    And so, since $\sum_{e\in E}\mu_{t,q}([e])$ converges absolutely, we have 
    $$
    \sum_{e\in E}\mu_{t,q}([e])|\phi_{e}'(\pi(\sigma(\rho_e)))|^{t-q}<\infty.
    $$
    Hence, $H_{\mu_{t,q}}(\mathcal{U}_E)<\infty$.
    
\end{proof}

By \cite[Theorem 17.4.7 and Theorem 17.5.1]{Urbanski22} we have the following corollary.

\begin{corollary}\label{cor:int}
    For the Hurwitz CGDS $\Phi$ and $(t,q)\in \mathsf{M}$, we have
    $$
        \int_{E_A^{\infty}}\Xi_{\Phi}d\mu_{t,q}>-\infty.
    $$
    Consequently $\mu_{t,q}$ is the unique equilibrium state for the potential $(t-q)\Xi_{\Phi}$.
\end{corollary}

\section{Complex Diophantine approximation}
In this section we relate the Lyapunov exponent of the Hurwitz conformal graph directed system to complex Diophantine approximation and prove Theorem \ref{thm:main-1}. In Section 3.1, 
we recall the Hurwitz complex continued fraction expansion and its basic properties. In Sections 3.2 and 3.3, we show that the exponential decay rate of the approximation error $\left|z-\frac{p_n}{q_n}\right|$ is encoded dynamically by the Lyapunov exponent of the Hurwitz conformal graph directed system. 
This allows us to apply the multifractal analysis developed in \cite{NDDis} to obtain the multifractal spectrum for complex Diophantine approximation.

\subsection{Complex continued fraction}
Any complex number $z \in \mathbb{C}$ can be written as the continued fraction
$$
z = a_0 + \cfrac{1}{a_1 +
      \cfrac{1}{a_2 +
      \cfrac{1}{a_3 + \dots}}} \quad,
$$
where $a_0,a_1,a_2,\dots \in \mathbb{Z}[i].$

For $n \ge 0$, define the sequence $\{p_n\}$ and $\{q_n\}$ by the following recurrence relations:
$$
\begin{cases}
p_{-1}:=1,\quad p_0:=a_0, \quad p_{n+1}:=a_{n+1}p_n+p_{n-1},\\
q_{-1}:=0,\quad q_0:=1, \quad q_{n+1}:=a_{n+1}q_n+q_{n-1}.
\end{cases}
$$

The following lemma is a standard result in the study of continued fraction algorithms. We omit the proof and refer the interested reader to \cite[Theorem 6.1]{DaniNogueira14} for more details.
\begin{lemma}
For $n \ge 0$, we have
$$\frac{p_n}{q_n} := a_0 + \cfrac{1}{a_1 +
      \cfrac{1}{a_2 + 
      \cfrac{1}{\dots+ \cfrac{1}{a_n}}}} \in \mathbb{Q}(i).$$
Moreover, we have $\frac{p_n}{q_n} \to z.$
\end{lemma}

Define also the $(n)$-th complete quotient $z_n$ as
$$z_0 := z, \text{ and } z_{n} := a_n+\frac{1}{z_{n+1}}.$$

\begin{lemma}\label{lem:number-theory}
Let $z \in \mathbb{C}$ have a Hurwitz complex continued fraction expansion, with convergents $p_n/q_n$. Assume that the limit
$
\ell := \lim_{n\to\infty} \frac{1}{n}\log |q_n|
$
exists and is finite. Then
$$
\lim_{n\to\infty} -\frac{1}{n}\log\left|z-\frac{p_n}{q_n}\right| = 2\ell.
$$
\end{lemma}

\begin{proof}
For Hurwitz convergents we have the exact error formula (see \cite[Proof of Proposition 3.3]{DaniNogueira14}) 
$$
\left|z-\frac{p_n}{q_n}\right|
=
\frac{1}{|q_n|^2\left|z_{n+1}+\frac{q_{n-1}}{q_n}\right|},
$$
where $z_{n+1}$ is the $(n+1)$-st complete quotient. 
Hence
$$
-\frac{1}{n}\log\left|z-\frac{p_n}{q_n}\right|
=
2\frac{\log|q_n|}{n}
+
\frac{1}{n}\log\left|z_{n+1}+\frac{q_{n-1}}{q_n}\right|.
$$
Thus it suffices to prove that
$
\frac{1}{n}\log\left|z_{n+1}+\frac{q_{n-1}}{q_n}\right| \to 0.
$ Using
$z_{n+1}=a_{n+1}+\frac{1}{z_{n+2}}$ and $q_{n+1}=a_{n+1}q_n+q_{n-1}$, we obtain
$$
z_{n+1}+\frac{q_{n-1}}{q_n}
=
\frac{q_{n+1}}{q_n}+\frac{1}{z_{n+2}}.
$$

In the Hurwitz algorithm, since the remainder $z_n-f(z_n)$ lies in the fundamental square $X$, we have
$|z_n-f(z_n)| \le \frac{1}{\sqrt{2}}$ and thus $|z_{n+1}| \ge \sqrt{2}$.
Also, for Hurwitz convergents we have $|q_{n+1}|>|q_n|$. 
Using
$
z_{n+1}+\frac{q_{n-1}}{q_n}
=
\frac{q_{n+1}}{q_n}+\frac{1}{z_{n+2}},
$
the reverse triangle inequality gives
$$
\left|z_{n+1}+\frac{q_{n-1}}{q_n}\right|
\ge
\left|\frac{q_{n+1}}{q_n}\right|-\left|\frac{1}{z_{n+2}}\right|
\ge
1-\frac{1}{\sqrt{2}},
$$
while the triangle inequality gives
$$
\left|z_{n+1}+\frac{q_{n-1}}{q_n}\right|
\le
\left|\frac{q_{n+1}}{q_n}\right|+\frac{1}{\sqrt{2}}
\le
2\left|\frac{q_{n+1}}{q_n}\right|.
$$
Therefore,
$
\frac{1}{n}\log\left(1-\frac{1}{\sqrt{2}}\right)
\le
\frac{1}{n}\log\left|z_{n+1}+\frac{q_{n-1}}{q_n}\right|
\le
\frac{\log 2}{n}+\frac{\log|q_{n+1}|-\log|q_n|}{n}.
$
The left-hand side tends to $0$. Since $\frac{1}{n}\log|q_n| \to \ell$,
we also have $\frac{1}{n}\log|q_{n+1}| \to \ell,$
and therefore
$
\frac{\log|q_{n+1}|-\log|q_n|}{n}\to 0.
$
By the squeeze theorem, we have
$$
\frac{1}{n}\log\left|z_{n+1}+\frac{q_{n-1}}{q_n}\right| \to 0.
$$
This completes the proof.
\end{proof}

\subsection{Continued fractions and diameters of cylinder sets in GDMS}
\begin{definition}
For a finite block $a_1\dots a_n\in E_A^n$, the {\it cylinder of depth $n$ associated to this block} is 
$
[a_1,\dots,a_n]
:=
\{\om\in E_A^{\infty}:\ \om_1=a_1,\dots,\om_n=a_n\}.
$
In other words, $\pi([a_1\dots a_n])$ is the set of all points in $X$ who have a Hurwitz coding whose first $n$ digits are $a_{1_1},\dots,a_{n_1}$.
\end{definition}

\begin{lemma}
Let $a_1\dots a_n\in E_A^n$, and let $p_k,q_k$ be the corresponding convergents.

Define the Möbius map
$$
M_n(w):=\frac{p_n w+p_{n-1}}{q_n w+q_{n-1}}.
$$
Then, up to boundary points,
$\pi([a_1\dots a_n])=M_n(X_{t(a_{n})})$.
\end{lemma}

\begin{proof}
Fix $z\in \pi([a_1\dots a_n])$, and let $w:=T^n(z)\in X.$
Since $z$ has a coding where the first $n$ digits are $a_{1_1},\dots,a_{n_1}$, a Hurwitz continued fraction expansion of $z$ may be written as
$$
z=\cfrac{1}{a_{1_1}+\cfrac{1}{a_{2_1}+\cfrac{1}{\ddots+\cfrac{1}{a_{n_1}+w}}}}.
$$
Unwinding this expression using the recurrence relations for $p_k,q_k$ gives
$$
z=\frac{p_n w+p_{n-1}}{q_n w+q_{n-1}}=M_n(w).
$$
Thus every $z\in \pi([a_1\dots a_n])$ belongs to $M_n(X_{t(a_{n})})$.

Conversely, if $w\in X_{t(a_{n})}$, then $z:=M_n(w)$
has a continued fraction expansion beginning with the digits $a_{1_1},\dots,a_{n_1}$, with a tail given by a coding of $w$. Note that if this coding is not unique then $w\in \partial X_{t(a_{n})}$. Hence, $z\in \pi([a_1,\dots,a_n])$.

The conclusion follows.
\end{proof}

\begin{lemma}\label{lem:diameter}
With the notation above, we have
$$
\operatorname{diam}(\pi([a_1\dots a_n]))\asymp |q_n|^{-2},
$$
where the implied constants are independent of $n$.
\end{lemma}

\begin{proof}
Since
$$
\pi([a_1\dots a_n])=M_n(X_{t(a_{n})}),
$$
it suffices to estimate the size of the image of $X_{t(a_{n})}$ under $M_n$.

The derivative of the Möbius map $M_n$ is
$
M_n'(w)
=
\frac{p_nq_{n-1}-p_{n-1}q_n}{(q_n w+q_{n-1})^2}.
$
Using the determinant identity $p_nq_{n-1}-p_{n-1}q_n=\pm 1$,
we obtain
$$
|M_n'(w)|=\frac{1}{|q_n w+q_{n-1}|^2}.
$$

We first derive an upper bound. Since $X_{t(a_{n})}\subset X$ and $X$ is bounded, there exists $R>0$ (e.g., $R=\sqrt{2}/2$) such that
$|w|\le R$ for all $w\in X_{t(a_{n})}$.
Hence
$
|q_n w+q_{n-1}|
\le |q_n|\,|w|+|q_{n-1}|
\le R|q_n|+|q_{n-1}|.
$
For Hurwitz convergents we have $|q_{n-1}|<|q_n|$, therefore $|q_n w+q_{n-1}|\le (R+1)|q_n|$.
It follows that
$
|M_n'(w)|\ge \frac{1}{(R+1)^2|q_n|^2}.
$

For the reverse bound, write
$
q_n w+q_{n-1}=q_n\left(w+\frac{q_{n-1}}{q_n}\right).
$
Since $w\in X$ and the ratios $q_{n-1}/q_n$ stay in a controlled region for the Hurwitz algorithm, there exists a constant $c>0$ such that $
\left|w+\frac{q_{n-1}}{q_n}\right|\ge c$ for all $w \in X$ away from the usual boundary ambiguities. Hence $|q_n w+q_{n-1}|\ge c|q_n|$, and therefore
$
|M_n'(w)|\le \frac{1}{c^2|q_n|^2}.
$

Thus $|M_n'(w)|\asymp |q_n|^{-2}$ for all $w\in X_{t(a_{n})}$.
Since $M_n:= \phi_{a_1}\circ\dots \circ\phi_{a_n}$ is an admissible composition of maps from $\Phi$, \cite[Lemmas 19.3.5]{Urbanski22} and Lemma \ref{lem:diam1}  imply
$
\operatorname{diam}(M_n(X))=\operatorname{diam}(\pi([a_1\dots a_n]))\asymp |q_n|^{-2}.
$
\end{proof}

\subsection{Complex Diophantine approximation rates}
\begin{proposition}\label{prop:connection}

For any $\om \in E_A^{\infty}$, we have 
$$
    \lambda(\om) = 2\lim_{n\to\infty} \frac{1}{n}\log|q_n| = -\lim_{n \to \infty} \frac{1}{n}\log \left|\pi(\om) - \frac{p_n}{q_n}\right|.
$$

\end{proposition}

\begin{proof}
We first observe that by Lemma \ref{lem:diameter}, we have
\begin{align*}
\lambda(\om)&:=-\lim_{n\to\infty}\frac{1}{n}\sum_{j=0}^{n-1}\log|\phi'_{\sigma^j(\om)_1}(\pi(\sigma^{j+1}(\om)))| \\
&= -\lim_{n \to \infty}\frac{1}{n}\log {\rm diam}(\pi([\om_1\dots \om_n]))= 2\lim_{n \to \infty} \frac{1}{n} \log |q_n|.  
\end{align*}

On the other hand, by Lemma \ref{lem:number-theory}, we have 
$$-\lim_{n \to \infty}\frac{1}{n}\log \left|\pi(\om) - \frac{p_n}{q_n}\right| = 2\lim_{n\to\infty}\frac{1}{n}\log |q_n|.$$
Therefore, the conclusion follows.
\end{proof}

Let $\lambda_{\rm min}$ be the minimal Lyapunov exponent. For any $\alpha \ge \lambda_{\rm min}$, let
$$X(\alpha):=\left\{z \in X : -\lim_{n \to \infty} \frac{1}{n}\log \left|z - \frac{p_n}{q_n}\right| = \alpha \right\}$$ 
be the collection of complex numbers in $X$ for which the asymptotic exponential rate of (complex) Diophantine approximation equals $\alpha$. Define the function $S \colon (\lambda_{\rm min}, \infty) \to \mathbb R$ by
$$S(\alpha):= {\rm HD}(X(\alpha))$$
where ${\rm HD}(X(\alpha))$ denotes the Hausdorff dimension of $X(\alpha)$. We will study $S(\alpha)$ by considering the following projected Lyapunov spectrum.

For a CGDMS with limit set $J$, the projected Lyapunov multifractal decomposition of $E_A^{\infty}$ is given by 
$$
    J(\alpha): = \pi(E_A^{\infty}(\alpha)) :=\pi(\{\om\in E_A^{\infty} : \lambda(\om)=\alpha\})
$$
since $\Phi$ is finitely irreducible, cofinitely regular, and satisfies the SOSC, one may, without loss of generality, only consider those points in $J(\alpha)$ with unique codings (see \cite[Proposition 6.5]{NDDis}). In either case, the Lyapunov specturm of $\Phi$ is given by 
$$
    t(\alpha):= {\rm HD}(J(\alpha)).
$$

\begin{remark}\label{rmk:3.7}
We observe that $t(\alpha) = S(\alpha)$ by Proposition \ref{lem:number-theory} as $X(\alpha)\setminus J(\alpha)$ is countable. 
\end{remark}

The following theorem is a special case of \cite[Theorem 3.1 (see also Theorem 7.1)]{NDDis} that describes the projected Lyapunov spectra of a CGDMS in terms of pressure.

\begin{theorem} \label{thm:main-3}
Let $\Phi =\{\phi_e\}_{e\in \NN}$ be a cofinitely regular finitely irreducible CGDMS satisfying the SOSC and let $\mathsf{M}$ be the Manhattan region for the Lyapunov pressure. If for all $(t,q)\in \mathsf{M}$, $-\int \Xi_\Phi d \mu_{t,q}<\infty$, then the projected Lyapunov spectrum of $\Phi$ is given by the unique real analytic solution $(t(\alpha),q(\alpha))$ to the system 
$$
\begin{cases} \cP(t,q) =q\alpha\\\frac{\partial \cP}{\partial q}(t,q) = \alpha \end{cases}
$$ on the interval $\alpha\in (\alpha_{\min},\infty)$ and
$$
t(\alpha) = \frac{1}{\alpha}(\cP(-q(\alpha))-q(\alpha)) = \frac{1}{\alpha}\inf_{q\in\RR}\{\cP(-q)-q\alpha\}, \forall \alpha\in (\alpha_{\min},\infty)
$$
where $\cP(q):= \cP(0,-q)$.
\end{theorem}

Now we give a proof of Theorem \ref{thm:main-1}.
\begin{proof}[Proof of Theorem \ref{thm:main-1}]
By Remark \ref{rmk:SOSC}, Corollary \ref{cor:int}, Lemma \ref{lem:cfr}, and Theorem \ref{thm:FI} the Hurwitz CGDS satisfies the hypotheses of Theorem \ref{thm:main-3}. By Remark \ref{rmk:3.7}, the conclusion follows.
\end{proof}

\section{Cusp excursions on the Bianchi orbifold}
In this section, we apply the multifractal analysis of Lyapunov exponent of the Hurwitz CGDS to cusp excursions on the Bianchi orbifold $\mathcal O =\mathbb H^3/\operatorname{PSL}(2,\mathbb Z[i])$ and prove Theorem \ref{thm:main-2}.
We first recall the relevant hyperbolic geometry in Section 4.1. In Section 4.2, we define the asymptotic average excursion time of a geodesic issuing from the cusp. We show that it is equal to the Lyapunov exponent of the Hurwitz CGDS and prove Theorem \ref{thm:main-2}.

\subsection{Hyperbolic geometry}
Consider the upper half-space model of the hyperbolic $3$-space, i.e.,
$$
\mathbb H^3 :=\{(w,t)\in \mathbb C\times \mathbb R_{>0}\},
$$
with boundary at infinity $\partial\mathbb H^3:=\widehat{\mathbb C}=\mathbb C\cup\{\infty\}$. The hyperbolic metric is $$ds^2=\frac{|dw|^2+dt^2}{t^2}.$$

Let $z\in\mathbb C$. An {\it oriented geodesic with forward endpoint $z$} is a semicircle or vertical line joining its two boundary endpoints. 

\medskip

Let $\Gamma:=\mathrm{PSL}(2,\mathbb Z[i])$. Then $\Gamma$ is a discrete subgroup of $\mathrm{PSL}(2,\mathbb C)$ with torsion elements, therefore the quotient $\mathcal{O}:= \mathbb H^3/\mathrm{PSL}(2,\mathbb Z[i])$ is an orbifold. Moreover, since $\Gamma$ is a {\it Bianchi group}, $\mathcal{O}$ is called a Bianchi orbifold.

The lift of the cusp of the Bianchi orbifold $\mathcal{O}$ in the universal cover $\mathbb H^3$ is the cusp at $\infty$, which is represented by high horoballs of the form
$$
\mathcal H_Y := \{(w,t)\in\mathbb H^3 : t\ge Y\}.
$$
We say that a geodesic makes a cusp excursion if, after applying a suitable element of $\mathrm{PSL}(2,\mathbb Z[i])$, the image geodesic rises to large height $t$.
Fix the cusp of $\mathcal O$ corresponding to $\infty$, and let $\Gamma_\infty<\Gamma$ denote its stabilizer.

\medskip

For each $z \in \mathbb C$, let $\widetilde\gamma_z\subset \mathbb H^3$ be the oriented geodesic with
endpoints $(\infty,z)$, oriented from $\infty$ to $z$, and let
$$
\gamma_z:=\pi(\widetilde\gamma_z)
$$
be its projection to $\mathcal O$. Thus each $z\in \mathbb C$ determines an oriented geodesic
in $\mathcal O$ issuing from the chosen cusp. If $\eta\in \Gamma_\infty$, then $\widetilde\gamma_{\eta(z)}=\eta(\widetilde\gamma_z)$. Hence
$\widetilde\gamma_{\eta(z)}$ and $\widetilde\gamma_z$ project to the same oriented geodesic in $\mathcal O$.
Therefore the assignment $z \mapsto \gamma_z$
depends only on the $\Gamma_\infty$-orbit of $z$.

\subsection{Asymptotic average excurison time}
In this section, we summarize Baumgartner--Pollicott's definition of {\it asymptotic average excurison time}; see \cite[Section 5]{baumgartnerP}.
Let $\gamma$ be a geodesic in $\mathcal{O}$ issued from the cusp.
Let $\tilde\gamma_0$ be a lift of $\gamma$ to $\mathbb H^3$ such that the repelling endpoint is $\infty$. 
Then $\tilde\gamma_0$ has endpoints $(z, \infty)$ with $z \in \mathbb C$. Let $\beta$ be the representative of $z$ in $X$, i.e., $\beta= [0,a_1,a_2,\dots] = [0,a_1(z),a_2(z),\dots].$

For $n \ge 1$, following Baumgartner--Pollicott \cite[Section 5.1]{baumgartnerP}, we define
$$
\tilde\gamma_n
:=
\bigl(T_{(-1)^{n-1}a_n}S\bigr)\circ \tilde\gamma_{n-1},
$$
where
$$
S=\begin{pmatrix}
0 & 1\\
-1 & 0
\end{pmatrix},
\qquad
T_q=\begin{pmatrix}
1 & q\\
0 & 1
\end{pmatrix}
\quad (q\in \mathbb Z[i]).
$$
The matrix $S$ performs the reciprocal step $z \mapsto -1/z$, while $T_q$ subtracts the chosen digit. Therefore, $T_{(\pm a_n)}S$ implement one continued-fraction step on the geodesic.

\medskip

The {\it intersection times} $t_n$ are defined by the condition $\tilde\gamma_n(t_n)\in H(0)$,
where $H(0)$ is the unit hemisphere.
Finally, the {\it excursion times} are defined as $$t_n^*(z):=\max_{1\le i\le n} t_i.$$

Baumgartner--Pollicott proved that the {\it asymptotic average excursion time} $\lim_{n \to \infty}\frac{t_n^*(z)}{n}$ satisfies the following property.

\begin{proposition}[{\cite[Proof of Proposition 5.1]{baumgartnerP}}]\label{prop:secion7BP}
We have
$$\lim_{n \to \infty}\frac{t_n^*(z)}{n} = \lim_{n \to \infty}\frac{2\log|q_n|}{n},$$ whenever these limits exist.
\end{proposition}

\begin{corollary}
For any $\om \in E_A^{\infty}$, we have
$$
\lambda(\om)=\lim_{n\to\infty}\frac{2\log |q_n|}{n} = \lim_{n \to \infty}\frac{t_n^*(\pi(\om))}{n},
$$
whenever these limits exist.
\end{corollary}\label{prop:BP24}

\begin{proof}
The conclusion follows from Proposition \ref{prop:connection} and Proposition \ref{prop:secion7BP}.
\end{proof}

\medskip

We give a proof of Theorem \ref{thm:main-2}.

\begin{proof}[Proof of Theorem \ref{thm:main-2}]
By Corollary \ref{prop:BP24}, we have $\hat S(\alpha) = S(\alpha)$. Then the conclusion follows from Theorem \ref{thm:main-1}.
\end{proof}

\section{Appendix}
In this Appendix, we describe some of the finer points regarding the Hurwitz CGDS. Let $D = \mathbb{Z}[i]\setminus\{0,\pm1,\pm i\}$ be as before, then we consider $D\times \cX$ where 
$$
    \cX := \{C_1,C_2,C_3,C_4,H_1,H_{i},H_{-1},H_{-i},P_1,P_2,P_3,P_4\}
$$
is the disjointification of the prototype cover of $X$ (see Figure \ref{fig:1}). The prototype cover $\{Q_z\}_{z\in D}$ consists of finitely many sets with the property that $\phi_z(Q_z)\subset \overline{X}$. These prototype sets are discussed in \cite{BGH23}. The sets in $\cX$ are regular closed, that is, $\overline{{\rm Int}_{\mathbb{R}^d} (X_{v})}=X_v$ for all $X_v\in \cX$. The maps $\Phi$ that comprise the Hurwitz CGDS are inversions derived from inverse branches of the Hurwitz map. As before, for any $z\in D$ and $(z,Y)\in D\times \cX$ we define $\phi_z(w) := \frac{1}{z+w}$ and $\phi_{(z,Y)}(w):=\frac{1}{z+w}\vert_Y$ for all $w\in\overline{X}$.
We note that since the Hurwitz map maps the fundamental domain $X$ to itself, $\Phi$ is not indexed by $D\times \cX$, but by a subset, as one observes via computation that for $\sqrt{2}\leq |z|<\sqrt{8}$, $\phi_{z}(X)\not\subseteq X$. As we will see, due to properties of the prototype regions we will also need to be careful about the case of $|z| = \sqrt{8}$ when collecting together the maps in $\Phi$. There are some maps for which we do not require many restrictions when restricting down to elements of $\cX$. Many of the arguments here are assisted by \cite[Figure 2]{BGH23} which we have reproduced here in Figure \ref{fig:Cylinders} for the convenience of the reader.

\begin{figure}[!ht]

\begin{center}
\includegraphics[scale=0.65,  trim={5cm 17cm 5cm 1.5cm},clip]{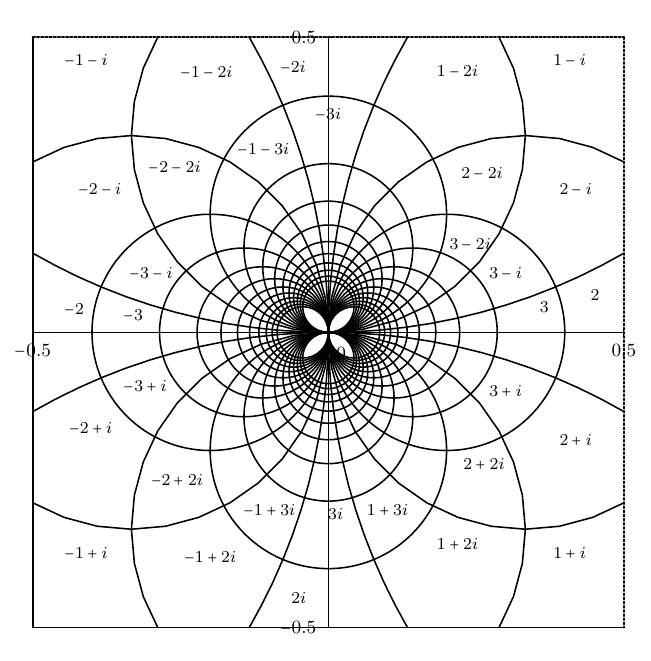}
\caption{ Realization of cylinders of level $1$. }\label{fig:Cylinders}
\end{center}
\end{figure}

\begin{proposition}\label{edges1}
 If $z = ai^k$ with $a\in \ZZ, |a|>2$ and $k\in\{0,1\}$, then $\phi_{(z,Y)}(Y)\subset H_{{\rm sgn}(a)i^k}$ for all $Y\in \cX$.
\end{proposition}
\begin{proof}
    This is immediate as the closed prototype set for $\phi_z$ is $\overline{X}$ and for each $z = ai^k$ with $|a|>2$ we have $\phi_z(X) \subseteq H_{{\rm sgn}(a)i^k}$. 
\end{proof}

\begin{proposition}\label{edges2}
If $z = a+bi$ with $a,b\in \ZZ\setminus \{0\}$ and $|z|>\sqrt{8}$ then for all $Y\in \cX$, $\phi_{(z,Y)}(Y)\subset P_j$ where 
$$
j = \begin{cases}
    1 & {\rm if}\;a>0, b<0\\
    2 & {\rm if}\;a<0, b<0\\
    3 & {\rm if}\;a<0, b>0\\
    4 & {\rm if}\;a>0, b>0\\
\end{cases}.
$$

\end{proposition}

\begin{proof}
    Once again, this is immediate as the prototype region for $\phi_z$ is $\overline{X}$ and for each of these $z$, $\phi_z(X)\subseteq P_i$ according to the cases stated in the proposition.
\end{proof}

For $|z|=\sqrt{8}$ we must be careful about which domain since $\phi_z(X)\cap C_i\neq \emptyset $ and $\phi_z(X)\cap P_i\neq \emptyset$ where again, $i$ is determined by the cases in the previous proposition. The prototype region of $|z| = \sqrt{8}$ is still the whole square, but using information about the prototype regions of $|z| = \sqrt{5}$ we can actually understand the co-domains of maps $\phi_{z,Y}$ rather quickly.

\begin{proposition}\label{edges3}
    Let $Q_{a+bi}$ be the prototype region for some $a+bi$ with $|a+bi| = \sqrt{5}$ and 
    $$
    z = \begin{cases}a+{\rm sgn}(a)+bi & {\rm  if}\; |a| = 1\\
    a+(b+{\rm sgn}(b))i & {\rm if}\; |a|\neq 1
    \end{cases}.
    $$
    Then for the unique $k$ such that ${\rm Int}(Q_{a+bi})\cap C_k = \emptyset$, we have 
    $$
    \begin{cases}
        \phi_{(z,C_k)}(C_k)\subset C_{k+1} \; {\rm if}\; k\in \{1,3\}\\ 
        \phi_{(z,C_k)}(C_k) \subset C_{k-1} \; {\rm if}\; k\in \{2,4\}
    \end{cases}
    $$
    and for $Y\in \cX\setminus{C_k}$ we have
    $$
        \begin{cases}
        \phi_{(z,Y)}(Y)\subset P_{k+1} \; {\rm if}\; k\in \{1,3\}\\ 
        \phi_{(z,Y)}(Y)\subset P_{k-1} \; {\rm if}\; k\in \{2,4\}
        \end{cases}.
    $$
\end{proposition}

\begin{proof}
    The image of $\phi_{z}$ on the whole square $\overline{X}$ is a shield cut by the circular arc in the $j^{th}$ quadrant separating $C_j$ from $C_j^c$ where $j = k\pm 1$ depending on $z$ as described in the statement of the proposition. The part of the shield interior to $P_j$ has $5$ bounding curves, and the part of the shield interior to $C_j$ has $3$ bounding curves. The result follows as the map $\phi_z$ is conformal and as the corner of the square in the $i^{th}$ quadrant is sent interior to $C_j$ by $\phi_{z}$. 
\end{proof}

For $|z| < \sqrt{8}$, we have even more bookkeeping to do with domains and ranges. To summarize the allowable domains $Y$ we have Table \ref{tab:1}.  

\begin{center}
\begin{table}[!ht]
\caption{Maps $\phi_{z,Y}$ for $|z|<\sqrt{8}$, $Y\in \cX$.}
\label{tab:1}
\begin{tabular}{|c|c|}
\hline
$z$ & $Y$\\
\hline
$1+i$ & $P_1,H_1,H_{-i},C_1$\\
\hline
$1-i$ & $P_4,H_1,H_i,C_4$\\
\hline
$-1+i$& $P_2,H_{-i},H_{-1},C_2$\\
\hline
$-1-i$& $P_3,H_{-1},H_i,C_3$\\
\hline
$1+2i,2+i$ & $\neq C_3$\\
\hline
$1-2i,2-i$ & $\neq C_2$\\
\hline
$-1+2i,-2+i$ & $\neq C_4$\\
\hline
$-1-2i,-2-i$ & $\neq C_1$\\
\hline
$2$ & $\neq P_2,P_3,H_{-1},C_2,C_3$\\
\hline
$-2$ & $\neq P_1,P_4,H_1,C_1,C_4$\\
\hline
$2i$ & $\neq P_3,P_4,H_i,C_1,C_4$\\
\hline
$-2i$ & $\neq P_1,P_2,H_{-i},C_1,C_2$\\
\hline

\end{tabular} 
\end{table}
\end{center}

The case of $|z|=\sqrt{2}$ is easiest.
\begin{proposition}\label{edges4}
    For $(z,Y)\in D\times\cX$ with $|z|=\sqrt{2}$ described by Table \ref{tab:1}, suppose $C_k$ is a possible $Y$ for this $z$. Then $\phi_{(z,Y)}(Y)\subset C_j$ where 
    $$
    j= \begin{cases}
        1& {\rm if\;} k=4\\
        2& {\rm if\;} k=3\\
        3& {\rm if\;} k=2\\
        4& {\rm if\;} k=1\\
    \end{cases}
    $$
\end{proposition}

\begin{proof} 
    For prototype region $Q_z$, $\phi_z(Q_z)\subsetneq C_j$ and the sets given in Table \ref{tab:1} for this $z$ are such that their union is $Q_z$.  
\end{proof}

Next we address the case of $|z|=2$ which is also quick, we omit the argument as it is similar to the $|z| = \sqrt{2}$ case.

\begin{proposition}\label{edges5}
    For $(z,Y)\in D\times\cX$ with $|z|=2$ described by Table \ref{tab:1}, then $\phi_{(z,Y)}(Y)\subset H_{\frac{z}{2}}$.  
\end{proposition}

The last and most complicated case is that of $|z|=\sqrt{5}$ which can be checked by computation noting the images of the $H$ sets and that $\phi_z$ is conformal.  We summarize this case in Table \ref{tab:2}.

\begin{proposition}\label{edges6}
    For a given $z$ and $Y$, the Table \ref{tab:2} describes  the associated co-domain $X_{i(z,y)}$ of $\phi_{z,Y}$. 

    \begin{center}
\begin{table}[!ht]
\caption{Maps $\phi_{(z,Y)}:Y\to X_{i(z,Y)}$ for $|z|=\sqrt{5}$}
\label{tab:2}
\begin{tabular}{|c|c|c|}
\hline
$z$ & $Y$ & $X_{i(z,y)}$ \\
\hline
$1+2i$ & $P_1,P_2,H_1,H_{-i},H_{-1},C_1,C_2$ & $P_4$ \\
\hline
$2+i$ & $P_1,P_4,H_1,H_{-i},H_{i},C_1,C_4$ & $P_4$ \\
\hline
\hline
$1+2i$ & $P_3,P_4,H_{i},C_4$ & $C_4$ \\
\hline
$2+i$ & $P_2,P_3,H_{-1},C_2$ & $C_4$ \\
\hline
\hline
$1-2i$ & $P_3,P_4,H_1,H_{-1},H_{i},C_3,C_4$ & $P_1$ \\
\hline
$2-i$ & $P_1,P_4,H_1,H_{-i},H_{i},C_1,C_4$ & $P_1$ \\
\hline
\hline
$1-2i$ & $P_1,P_2,H_{-i},C_1$ & $C_1$ \\
\hline
$2-i$ & $P_2,P_3,H_{-1},C_3$ & $C_1$ \\
\hline
\hline
$-1+2i$ & $P_1,P_2,H_1,H_{-i},H_{-1},C_1,C_2$ & $P_3$ \\
\hline
$-2+i$ & $P_2,P_3,H_{-i},H_{-1},H_{i},C_2,C_3$ & $P_3$ \\
\hline
\hline
$-1+2i$ & $P_3,P_4,H_{i},C_3$ & $C_3$ \\
\hline
$-2+i$ & $P_1,P_4,H_{1},C_1$ & $C_3$ \\
\hline
\hline
$-1-2i$ & $P_3,P_4,H_1,H_{-1},H_{i},C_3,C_4$ & $P_2$ \\
\hline
$-2-i$ & $P_2,P_3,H_{-i},H_{-1},H_{i},C_2,C_3$ & $P_2$ \\
\hline
\hline
$-1-2i$ & $P_1,P_2,H_{-i},C_2$ & $C_2$ \\
\hline
$-2-i$ & $P_1,P_4,H_{1},C_4$ & $C_2$ \\
\hline
\end{tabular} 
\end{table}
\end{center}
\end{proposition}

\newpage
\bibliographystyle{plain}
\bibliography{VDM.bib}
 
\end{document}